\documentclass[12pt]{amsart}

\usepackage{bigints}
\usepackage{natbib}
\usepackage{graphicx}
\usepackage{dsfont}
\usepackage{mathrsfs}
\usepackage{mathptmx}
\usepackage{helvet}
\usepackage{courier}
\usepackage{multicol}
\usepackage{footmisc}
\usepackage{enumerate}
\usepackage{natbib}

\usepackage{float}
\restylefloat{table}
\restylefloat{figure}

\usepackage{xcolor}
\usepackage{color}

\usepackage{pgf,tikz}
\usepackage{pgfplots}
\usetikzlibrary{automata}
\usetikzlibrary{arrows}
\usetikzlibrary{positioning}

\usepackage{amsthm}

\theoremstyle{plain}

\newtheorem{Th}{Theorem}[section]

\newtheorem{Ex}[Th]{Example}

\newtheorem{Cor}[Th]{Corollary}
\newtheorem{Def}[Th]{Definition}
\newtheorem{Rem}[Th]{Remark}

\usepackage{bm}

\usepackage{tensor}
\usepackage{comment}

\pgfplotsset{compat=1.10}

\newcommand{\Exp}{\mathds{E}}

\newcommand{\ii}{\mbox{i}}
\newcommand{\Prob}{\mathds{P}}

\newcommand{\vect}[1]{\vec{#1}}
\renewcommand{\vect}[1]{\boldsymbol{#1}}

\newcommand{\mat}[1]{\boldsymbol{#1}}

\usepackage[top=1.75cm,bottom=1.75cm,left=2.5cm,right=2.5cm]{geometry}

\title{Inhomogeneous phase--type distributions and heavy tails}

 \author[Albrecher and Bladt]{Hansj\"org Albrecher and Mogens Bladt}
 \address[Albrecher]{Universit\'e de Lausanne, Quartier UNIL-Chamberonne, B\^atiment Extranef,
1015 Lausanne}
\email{hansjoerg.albrecher@unil.ch}
 \address[Bladt]{Department of Mathematical Sciences, University of Copenhagen, Universitetsparken 5, DK-2100 Copenhagen \O, Denmark}
 \email{bladt@math.ku.dk}

\begin{document}\maketitle

\allowdisplaybreaks
\begin{abstract}
We extend the construction principle of phase-type (PH) distributions to allow for inhomogeneous transition rates and show that this naturally leads to direct probabilistic descriptions of certain transformations of PH distributions. In particular, the resulting matrix distributions enable to carry over fitting properties of PH distributions to distributions with heavy tails, providing a general modelling framework for heavy-tail phenomena. We also illustrate the versatility and parsimony of the proposed approach for the modelling of a real-world heavy-tailed fire insurance dataset. 
    	
\end{abstract}
\section{Introduction}
A phase--type distribution (PH) is the distribution of the time until absorption in a Markov jump process with finitely many states of which one is absorbing and the remaining being transient. PH distributions have a long history in Applied Probability (see \cite{bladt2017matrix} and references therein), especially in areas such as insurance risk and queueing theory, where under PH assumptions exact solutions and explicit formulae can often be derived even in rather complex models (see e.g.\ \cite{asm:03,AsAl10}). Though the class of phase--type distributions is dense in the class of distributions on the positive real line (in the sense of weak convergence), by construction PH tails are always light (exponential decay), and the latter property is a main concern regarding applications where heavy tails are present (such as large claims in insurance). 

In \cite{bladt-nielsen-samorodnitsky:2015} the class of PH distributions was extended to allow for infinite-dimensional matrices which led to a class of distributions with phase--type like properties and a genuinely heavy tail, and the estimation of such distributions was considered in 
\cite{BladtNandayapa2018}. 

In this paper we take another approach to define a tractable and dense class of heavy-tailed distributions. The idea is to transform PH distributions into heavy-tailed ones, but rather than transforming the PH random variable directly, we transform the time scales of each state of the underlying Markov process, leading to time--dependent jump intensities. This time-scaling allows to carry over some of the computational tools and advantages of the PH class outside of the latter and leads to a probabilistic description of classes of heavy-tailed distributions akin to the one for PH distributions. Compared to the approach in \cite{bladt-nielsen-samorodnitsky:2015}, the present approach has the advantage of preserving the finite dimensionality, and offers flexibility and transparency in the choice of tail type. A distinctive disadvantage is that renewal theory and, consequently, ladder height techniques break down due to the inhomogeneity of time. However, we show that for certain transformations calculations are still explicit, and the flexibility of PH distributions can in this way be more efficiently transferred to heavy-tailed distributions than with previous approaches. Moreover, the resulting new class of matrix distributions is again dense in the class of all distributions on the positive half-line, which paves the way for using it in model fitting. \\

From a modelling perspective, the approach proposed in this paper has some attractive features. Let us recall that a particular strength of the class of PH distributions is that it extends the favourable properties of an exponential random variable (probabilistic ones such as lack of memory as well as analytic ones such as the simplicity of calculations due to the elementary and paramount role of the exponential function in real and complex analysis) to cases where the exponential random variable itself is not a good model. It does so by concatenating these properties in a Markov jump process framework, leading to matrix expressions, but still maintaining those favourable properties. The denseness of the PH class then in principle allows to approximate any random variable and its role in the respective model with these techniques. On the other hand, Pareto random variables with their heavy tails can be motivated in various ways for particular modelling applications. One may for instance interpret them in terms of scaling phenomena in nature (like with considerations of Zipf's law). One can also simply view them as model candidates with power tails, which decay slower than exponential. Indeed, in various applications such a tail behaviour can be observed in data. Alternatively, one may also just see them as (possibly rescaled and shifted) exponentials of exponential random variables, which inherits some attractive mathematical properties, like a scaled version of lack of memory (and the latter is one reason for the particular attractiveness of Pareto random variables as modelling tools in certain applications, like reinsurance). At the same time, a pure Pareto random variable is often not a good fit to data (particularly when the entire range is considered). One way to deal with this in practice is to use spliced models (also sometimes called composite models), where one type of model is used for smaller values and a Pareto model is used for the tail (see e.g.\ \cite{pig,reynIme,scollnik2007}). 

A number of known distributions can be obtained as the distribution of a transformed exponential random variable. This is for example indeed the case for the Pareto, the Weibull and the generalized extreme value distribution (GEV). The obvious idea is to propose a matrix version of these distributions by transforming a phase--type distributed random variable with the same transformations. We show that the resulting distributions have similar forms and properties as their original counterparts, and that they may be seen as inhomogeneous phase--type distributions. This provides a unifying framework for seeing transformed phase--type distributed random variables as absorption times.  

Transformed phase--type random variables have e.g. been treated in \cite{ahn2012}, where the logarithmic transformation is used. This results in a Pareto type of tail behaviour. For other transformations we may obtain tail behaviours of other types like e.g. Weibull or GEV. The class of transformed distributions of a certain type is again dense in the class of distributions on the positive reals. In fact we will provide different sub--classes generated by mixtures of transformed Erlang distributions which are dense and can be written in an explicit and simple form.
In principle a classical phase--type distribution may approximate any heavy--tailed positive distribution arbitrarily well, but since they are light--tailed, the approximating distribution will always have a distinct tail which can become an important issue if applied to situations where the tail behaviour matters (like e.g. ruin probabilities). If one instead fits a transformed phase--type distribution to data with the "correct" tail, then not only will one be able to capture the proper tail behaviour but it will also lead to more parsimonious models (with fewer components) than fitting with the traditional PH class. On the other hand the transformed phase--type distributions may have computational advantages (and possibly methodological advantages in terms of interpretability) to other competitors like random variables with more general regularly varying tails.


From a statistical point of view one can apply standard techniques from phase--type fitting to the transformed random variables. This is for instance the case for the log--transformed phase--type distributions where one simply fits a phase--type distribution to the logarithm of the data points. In \cite{ahn2012}, a log--phase--type distribution was fitted to Danish fire insurance data. 
We provide an example using Dutch insurance data, where the main purpose is a comparison of our approach to a previously obtained model in \cite{abt}  which was based on splicing Erlang and Pareto distributions. Since the splicing components are in some sense implicit in the log--Phase--type setup, it is interesting to see whether we shall be able to retrieve parts of that model (and in a more automated way).

 While our estimation method is based on maximum likelihood which treats each data point with the same weight, it may be interesting in future research to extend the respective statistical techniques to incorporate extreme value analysis considerations for the fitting procedure of this type of models, with more weight being given to larger data points.

The paper is organized as follows. In Section \ref{sec2} we derive various theoretical properties of inhomogeneous PH distributions. On the basis of this representation, Section \ref{sec3} gives various properties for the class of matrix-Pareto distributions and provides some details for several special cases. Subsequently, Section \ref{sec:M-Gumbel} considers other transformations of PH distributions and studies some properties for the resulting distribution classes. In Section \ref{sec:Modelling} we then discuss in more detail the modelling dimension of matrix-Pareto distributions and illustrate its use for a set of Dutch fire insurance data that was recently studied by other statistical techniques in \cite{abt}. Finally, Section \ref{concl} concludes.

\section{Inhomogeneous Phase--type distributions}\label{sec2}
Consider a time--inhomogeneous Markov jump process (\cite{goodman-johansen1973}) $\{ X_t \}_{t\geq 0}$ on the finite state--space $E=\{1,2,...,p,p+1\}$, where states $1,2,...,p$ are transient and $p+1$ absorbing. Thus the intensity matrix of $\{ X_t \}_{t\geq 0}$
is of the form
\[  \mat{\Lambda}(t) = 
\begin{pmatrix}
\mat{T}(t) & \mat{t}(t) \\
\mat{0} & 0 
\end{pmatrix}  . 
  \]
Here, for any time $t\ge 0$, $\vect{t}(t)=-\mat{T}(t)\vect{e}$, where $\vect{e}$ denotes the column vector $\vect{e}=(1,1,...,1)^\prime$. Assume that $\Prob (X_0=p+1)=0$ and define 
 $\vect{\pi}=(\pi_1,...,\pi_p)$. 
 \begin{Def}
 Let 
 \[ \tau = \inf \{ t\geq 0 : X_t=p+1 \}  \]
 denote the time until absorption of $\{ X_t\}$. 
 The distribution of $\tau$ is then said to be an inhomogeneous phase--type distribution with representation $(\vect{\pi},\mat{T}(t))$ and we write 
 $\tau \sim \mbox{IPH}(\vect{\pi},\mat{T}(t))$. 
 \end{Def}

 The transition matrix $\mat{P}(s,t)=\{ p_{ij}(s,t) \}$, where
 \[  p_{ij}(s,t)=\Prob (X_t=j | X_s=i) , \]
 is related to the intensity matrix in terms of the product integral (see \cite{Gill-Johansen-1990,Slavik2007})
 \[  \mat{P}(s,t)= \prod_{s}^t (\mat{I}+\mat{\Lambda}(u)du) , \]
 which is defined by 
 \[ \prod_{s}^t (\mat{I}+\mat{\Lambda}(u)du) = \mat{I} + \sum _ { k = 1} ^ { \infty } \int _ { s } ^ { t } \int _ { s } ^ { u _ { k } } \cdots \int _ { s } ^ { u _ { 2} } \mat{\Lambda} \left( u _ { 1 } \right) \cdots  \mat{\Lambda} \left( u _ { k} \right) d u _ { 1} \cdots \text{d} u _ { k }  . \] 
 It is then straightforward to see that
 \begin{equation}
 \mat{P}(s,t) = 
\begin{pmatrix}
\displaystyle \prod_{s}^t (\mat{I}+\mat{T}(u)du) & \vect{e}-\displaystyle  \prod_{s}^t (\mat{I}+\mat{T}(u)du)\vect{e} \\
\vect{0} & 1
\end{pmatrix}  ,
  \label{eq:transition-for-PH}
 \end{equation}
 where $\vect{e}=(1,1,...,1)^{\prime}$.
\begin{Th}
Assume that $\tau \sim \mbox{IPH}(\vect{\pi},\mat{T}(t))$. Then the density $f$ and the distribution function $F$ of $\tau$ are given by 
\begin{eqnarray*}
f(x)&=&  \vect{\pi}\prod_{0}^x (\mat{I}+\mat{T}(u)du)\vect{t}(x) \\
F(x)&=& 1 - \vect{\pi}\prod_{0}^x (\mat{I}+\mat{T}(u)du)\vect{e} .
\end{eqnarray*}
\end{Th}

\begin{proof}
  Since $(\vect{\pi},0) \mat{P}(0,t)$ is the distribution of $X_t$, we get from  \eqref{eq:transition-for-PH} that
 \[ \Prob (X_t=j) = \vect{\pi}\prod_{0}^t (\mat{I}+\mat{T}(u)du)\vect{e}_j .\]
 The density $f$ for $\tau$ is now readily derived as
 \begin{eqnarray*}
 f(x)dx &=& \Prob (\tau\in (x,x+dx]) \\
 &=& \sum_{j=1}^p \vect{\pi}\prod_{0}^x (\mat{I}+\mat{T}(u)du)\vect{e}_j t_j(x)dx \\
 &=&  \vect{\pi}\prod_{0}^x (\mat{I}+\mat{T}(u)du)\vect{t}(x)dx ,
 \end{eqnarray*}
 so that
\begin{equation}
  f(x) = \vect{\pi}\prod_{0}^x (\mat{I}+\mat{T}(u)du)\vect{t}(x) . \label{eq:density}
\end{equation}
The distribution function $F$ for $\tau$ is obtained by 
 \begin{eqnarray*}
  1-F(x) &=& \Prob (\tau >x) \\
  &=&\Prob (X_x \in \{ 1,2,...,p\} ) \\
  &=&\sum_{j=1}^p \Prob (X_x=j) \\
  &=& \vect{\pi}\prod_{0}^x (\mat{I}+\mat{T}(u)du)\vect{e} .
\end{eqnarray*}
\end{proof}
By an entirely similar argument we obtain the following.
\begin{Th}\label{th:gen-vershoot}
 If $\tau \sim \mbox{IPH}(\vect{\pi},\mat{T}(t))$ then 
 \[ \Prob (\tau >s+t | \tau>s) = \frac{\vect{\pi} \prod_{0}^s(1+\mat{T}(u)du)}{\vect{\pi} \prod_{0}^s(1+\mat{T}(u)du)\vect{e}}\prod_s^t(1+\mat{T}(u)du)\vect{e}  \]
so that
\[  \tau-s | \{\tau>s \} \sim \mbox{IPH}\left( \vect{\alpha}, \mat{S}(\cdot)   \right) ,\]
where $\mat{S}(u)=\mat{T}(s+u)$ and
\[ \vect{\alpha}= \frac{\vect{\pi} \prod_{0}^s(1+\mat{T}(u)du)}{\vect{\pi} \prod_{0}^s(1+\mat{T}(u)du)\vect{e}}  .\]
\end{Th}
Hence the overshoot over a certain level is again phase--type distributed.
Methods for numerical evaluation of transition probabilities given by the product integral is treated in e.g. \cite{HELTON1976410} or \cite{max-moller-1992}. We shall, however, concentrate on the cases where the matrices $\mat{T}(u)$ commute, which will simplify expressions and numerical methods considerably.
\begin{Cor}
If  $\tau \sim \mbox{IPH}(\vect{\pi},\mat{T}(t))$ and
$\mat{T}(t_1)$ and $\mat{T}(t_2)$ commute for all $t_1,t_2\geq 0$, then the density $f$ and the distribution function $F$ of $\tau$ are given by 
\begin{eqnarray*}
 f(x) &=& \vect{\pi}\exp \left( \int_0^x \mat{T}(u)du \right)\vect{t}(x) \\
 F(x) &=& 1 - \vect{\pi}\exp \left( \int_0^x \mat{T}(u)du \right)\vect{e} .
\end{eqnarray*}
\end{Cor}
\begin{proof}
If $\mat{T}(t_1)$ and $\mat{T}(t_2)$ commute for all $t_1,t_2\geq 0$, then 
\begin{eqnarray*}
  \int _ { s } ^ { t } \int _ { s } ^ { u _ { k } } \cdots \int _ { s } ^ { u _ { 2} } \mat{T} \left( u _ { k } \right) \cdots  \mat{T} \left( u _ { 1} \right) d u _ { 1} \cdots d u _ { k }   &=& \frac { 1} { k ! } \int _ { s } ^ { t } \int _ { s } ^ { t } \cdots \int _ { s } ^ { t }  \mat{T} \left( u _ { k } \right) \cdots  \mat{T} \left( u _ { 1} \right) d u _ { 1} \cdots d u _ { k } \\
&=& \frac { 1} { k ! } \left( \int_s^t \mat{T}(u)du  \right)^k,
   \end{eqnarray*} 
so
\[ \prod_{s}^t (\mat{I}+\mat{T}(u)du) = \exp \left( \int_s^t \mat{T}(u)du \right)  \]
from which the result follows.
\end{proof}
An important special case where all matrices $\mat{T}(\cdot)$ commute, are the class of sub--intensity matrices on the form  $\mat{T}(t) = \lambda (t) \mat{T}$, which corresponds to the case where all transition rates of the Markov chain are scaled in the same way. More generally, block--diagonal matrices of the form
\[  \begin{pmatrix}
\mat{S}(u) & \mat{0} \\
\mat{0} & \mat{T}(u) 
\end{pmatrix} \]
commute whenever $\{ \mat{S}(u) \}$ commute and $\{ \mat{T}(u) \}$ commute. This includes the case where one of the two families of matrices are constant, e.g. $\mat{S}(u)=\mat{S}$ for all $u\geq 0$. In the latter case, the resulting  distribution will be a mixture of an inhomogeneous and a homogeneous phase--type distribution.
\begin{Def}
If $\mat{T}(t) = \lambda (t) \mat{T}$,
where $\lambda (t)$ is some known non--negative real function and  $\mat{T}$ is a sub--intensity matrix,
then we shall write
$\tau \sim \mbox{IPH}(\vect{\pi},\mat{T},\lambda )$ instead of 
$\mbox{IPH}(\vect{\pi},\mat{T}(t))$.
\end{Def}
\begin{Cor}\label{cor:PH-lambda-dens-and-dist}
If $\tau \sim \mbox{IPH}(\vect{\pi},\mat{T},\lambda )$, then the density $f$ and the distribution function $F$ of $\tau$ are given by 
\begin{eqnarray*}
 f(x) &=& \lambda (x) \vect{\pi}\exp \left( \int_0^x \lambda (t)dt\ \mat{T} \right)\vect{t}, \\
 F(x)&=&1- \vect{\pi}\exp \left( \int_0^x \lambda (t)dt\ \mat{T} \right)\vect{e}.
\end{eqnarray*}
\end{Cor}
\begin{proof}
Follows immediately from
\[ \vect{t}(t) = -\mat{T}(t)\vect{e} = -\lambda (t) \mat{T}\vect{e} = \lambda (t) \vect{t}  . \]
\end{proof}

\begin{Rem}
The commutativity condition of the matrices $\mat{T}(t)$ may be slightly relaxed
by assuming only quasi--commutativity of the family $\{ \mat{T}(t)\}$,  which is
characterized by the property
\[  \frac{d}{dt} \left( \int_s^t \mat{T}(u)du \right)^k = \frac{1}{k} \left( \int_s^t \mat{T}(u)du \right)^{k-1} \mat{T}(t) , \] 
see \cite[Section 1.4]{breuer2012markov}. In this case the above expressions for the density and the distribution function remain valid. 
\end{Rem}

If $\tau \sim \mbox{IPH}(\vect{\pi},\mat{T},\lambda )$ and $\lambda (t) \equiv 1$, then $\tau$ has a conventional phase--type distribution (see \cite{bladt2017matrix}), in which case we write $\tau\sim \mbox{PH}(\vect{\pi},\mat{T})$. Now consider $X\sim \mbox{PH}(\vect{\pi},\mat{T})$, $g:\mathds{R}_+\rightarrow \mathds{R}_+$ an increasing and differentiable function and $Y:=g(X)$. 
Then
\begin{eqnarray*}
\bar{F}_Y(y)&=&1-F_Y(y) = \Prob (g(X)> y ) \\
&=&\Prob (X> g^{-1}(y))  \\
&=&\vect{\pi}e^{\mat{T}g^{-1}(y)}\vect{e} 
\end{eqnarray*}
and consequently, using that $\vect{t}=-\mat{T}\vect{e}$,
\begin{equation}
f_Y(y) = \frac{d}{dy}F_Y(y) = \vect{\pi}e^{\mat{T}g^{-1}(y)}\vect{t}  \frac{1}{g^\prime (g^{-1}(y))} .\label{eq:dens} \end{equation}
This suggests that a random variable
$\tau\sim \mbox{IPH}(\vect{\pi},\mat{T},\lambda )$ can be obtained as a transformation of a phase--type distributed random variable $X\sim \mbox{PH}(\vect{\pi},\mat{T})$ as follows.
\begin{Th}\label{th:link-lambda-and-g}
Let $\tau \sim \mbox{IPH}(\vect{\pi},\mat{T},\lambda )$ and let $g$ be defined in terms of its inverse function through
\[  g^{-1} (x) = \int_0^x \lambda (t)dt .  \]
Then 
\[  \tau \stackrel{d}{=} g(X) , \]
where $X\sim \mbox{PH}(\vect{\pi},\mat{T})$. 
\end{Th}
\begin{proof}
First we notice that $g$ is well-defined since $\lambda (t)>0$ for all $t$ and correspondingly $x\rightarrow \int_0^x \lambda (t)dt$ is strictly increasing. Then consider 
$Y=g(X)$. We have that
\begin{eqnarray*}
 \Prob (Y>y)&=& \Prob \left( X> g^{-1}(y)  \right) \\
 &=& \Prob \left( X> \int_0^y \lambda (t)dt  \right) \\
 &=& \vect{\pi} \exp \left( \int_0^y \lambda (t)dt \mat{T} \right) \vect{e} .
 \end{eqnarray*} 
 Thus $Y\sim \tau$.
\end{proof}
If $h$ is an analytic function and  $\mat{A}$ is a matrix, we define
\[  h(\mat{A}) = \frac{1}{2\pi \ii} \oint_\gamma h(z)(z\mat{I}-\mat{A})^{-1}dz , \]
where $\gamma$ is a simple path enclosing the eigenvalues of $\mat{A}$. We refer to 
Section 3.4. of \cite{bladt2017matrix} for further details.
\begin{Th}
Let $\tau \sim \mbox{IPH}(\vect{\pi},\mat{T},\lambda )$, with $g$ defined in terms of its inverse function through 
\[  g^{-1} (x) = \int_0^x \lambda (t)dt .  \]
Assume that the Laplace transform for $g$, $L_g(s)$, exists for all $s>max_i \mbox{Re}(\lambda_i)$, where $\lambda_i$ denote the eigenvalues for $\mat{T}$. Then the mean of $\tau$ is given by 
 \begin{equation}
  \Exp (\tau ) = \vect{\pi} L_g(-\mat{T})\vect{t} . \label{eq:gen-mean-PH} 
  \end{equation}
More generally, if $g^\alpha$ exists in the same domain as above, where
 $\alpha>0$, then 
\begin{equation}
 \Exp (\tau^\alpha ) = \vect{\pi} L_{g^\alpha}(-\mat{T})\vect{t} ,  \label{eq:alpha-moment}
\end{equation}
where $L_{g^\alpha}$ denotes the Laplace transform of $g^\alpha$.
\end{Th}
\begin{proof}
Laplace transforms are analytic functions in the domain where they exist, so
\[  \Exp (\tau)=\Exp g(X) = \int_0^\infty g(x)\vect{\pi}e^{\mat{T}x}\vect{t} dx = \vect{\pi}L_g (-\mat{T})\vect{t} . \]
Similarly, if $L_{g^\alpha}$ exists, then it is analytic and 
\[ \Exp (\tau^\alpha)=\Exp g^\alpha (X) = \vect{\pi}L_{g^\alpha} (-\mat{T})\vect{t} . \]
\end{proof}

We notice the following trivial but important result. 
\begin{Th}\label{triv}
Suppose that 
\[ \tau_i \sim g(X_i) \]
where $X_1,...,X_N$ are independent and $X_i\sim \mbox{PH}(\vect{\pi},\mat{T})$. Let $X$ be distributed as the mixture of $X_1,...,X_N$, with density 
\[ f_X(x)=\sum_{i=1}^N \alpha_i f_{X_i}(x), \]
where $\alpha_1+...+\alpha_N=1$. Then the mixture of $\tau_1,...,\tau_N$ with probabilities $\alpha_1,...,\alpha_N$ is distributed as $g(X)$.
\end{Th}
\section{Matrix--Pareto distributions}\label{sec3}
We start by proving the following result.
\begin{Th}\label{th:F-and-f-of-log-PH}
Let  $Z=\exp (X) -1$ where $X\sim \mbox{PH}(\vect{\pi},\mat{T})$. Let $F_Z,\bar{F}_Z$ and $f_Z$ denote the distribution function, survival function and density of $Z$, respectively. Then
\begin{eqnarray*}
\bar{F}_Z(z)&=&1-F_Z(z)=\vect{\pi}(z+1)^{\mat{T}}\vect{e} \\
f_Z(z)&=&\vect{\pi}(z+1)^{\mat{T}-\mat{I}}\vect{t}
\end{eqnarray*}
If the real part of the eigenvalue  for $\mat{T}$ with maximum real part is smaller than $-\alpha$, then 
\[  \Exp ((Z+1)^\alpha) =  \vect{\pi}\left( -\alpha\mat{I}-\mat{T} \right)^{-1}\vect{t} <\infty. \]
The Laplace transform of $Z$ is given by
\[  L_Z(s) =  e^{s}\vect{\pi}s^{-\mat{T}}\Gamma (\mat{T},s)\vect{t} ,\ \ s>0, \]
where 
\[  \Gamma (\alpha,s) = \int_s^\infty t^{\alpha-1}e^{-t}dt \]
is the upper incomplete $\Gamma$--function. 
\end{Th}
\begin{proof}
Assuming that $z>0$, 
 $x\rightarrow x\log(1+z)$ is analytic in the whole complex plane and so is then its exponential $x\rightarrow (1+z)^x$. Therefore $(1+z)^{\mat{T}}$
 can be written in terms of the Cauchy--integral formula
 \[  (1+z)^{\mat{T}}=\frac{1}{2\pi\ii}\oint_\gamma (1+z)^y (y\mat{I}-\mat{T})^{-1}dy  \]
 where $\gamma$ is a simple path in $\mathds{C}$ enclosing the eigenvalues for $\mat{T}$.
The results then follow from Theorem \ref{th:link-lambda-and-g} with $\lambda (y)={1}/(1+y)$ (which indeed results in $g(x)=e^x-1$) and Corollary \ref{cor:PH-lambda-dens-and-dist}. 

The mean is straight-forward, as the mean of $e^{X}$ is simply the moment generating function of $X$ evaluated at 1.
The $\alpha$--moment of $Y+1$ is obtained by noting that the Laplace transform of $e^{\alpha x}$ is $s\rightarrow 1/(s-\alpha )$ and using \eqref{eq:alpha-moment}.
Concerning the Laplace transform,
\begin{eqnarray*}
L_Z(s)&=&\int_0^{\infty}e^{-sz} f_Z(z)dz \\
&=&\vect{\pi}\int_0^{\infty}(1+z)^{\mat{T}-\mat{I}}e^{-sz}dz \ \vect{t} \\
&=&e^{s}\vect{\pi}s^{-\mat{T}}\Gamma (\mat{T},s)\vect{t} ,
\end{eqnarray*}
since  $\alpha\rightarrow \Gamma (\alpha ,s)$ is an entire function for $s\neq 0$ (see \cite{Olver:2010:NHM:1830479}, 8.2)
\end{proof}


\begin{Rem}\rm
In \cite{ahn2012} similar results are proved for the distribution of $\exp (X)$ (not subtracting 1) apart from the Laplace transform. Their results are equivalent, though the presentation differs in notation and our use of functional calculus in the proof. 
\end{Rem}

\begin{Rem}\rm
If the real part of the eigenvalue for $\mat{T}$ with maximum real part is smaller than $-1$,  then the mean of $Z=e^{X}-1$ exists and is given by 
\[  \Exp (Z)= \vect{\pi}\left( -\mat{I}-\mat{T} \right)^{-1}\vect{t}-1 .\]
\end{Rem}

\begin{Rem}\rm
The matrix $\Gamma (\mat{T},s)$ is given by
\[ \Gamma (\mat{T},s) = \frac{1}{2\pi\ii} \int_\gamma \Gamma (z,s)(z\mat{I}-\mat{T})^{-1}dz ,   \]
where $\gamma$ is a simple path enclosing the eigenvalues for $\mat{T}$. If all eigenvalues are simple (i.e. have multiplicity one), then $ \Gamma (\mat{T},s)$ can be readily calculated by the residual theorem, as the terms only involve $\Gamma (\lambda_i,s)$ divided by polynomials. On the other hand, if an eigenvalue $\lambda$ has multiplicity $m>1$, then also 
the $(m-1)$-order derivative of a polynomial times $\Gamma (z,s)$ w.r.t. $z$ will have to be evaluated at $\lambda$, which in turn involves derivatives of $z\rightarrow \Gamma (z,s)$ of orders $1,...,m-1$.
\end{Rem}
While the distribution of $\exp (X)$ can naturally be called a log--phase--type distribution, we shall refer to distributions of $\exp (X)-1$ as \textit{matrix--Pareto distributions}, since the distribution of $Z$ from Theorem \ref{th:F-and-f-of-log-PH} may be seen as a Pareto distribution with a matrix parameter. A previous example conforming with this idea relates to the class of matrix--exponential distributions (containing the phase--type distributions), which bears its name from the fact that its elements can be seen as exponential distributions with a matrix parameter. 

To gain generality, and in particular in order to draw on the analogy with the generalized Pareto distribution $G_{\xi,\beta}(x)=1-(1+\xi x/\beta)^{-1/\xi}$ for $\xi>0$ (cf.\ \cite{ekt}), we introduce the following additional scaling in the transformation (for reasons that become clear in Theorem \ref{th:main-Pareto}, it is useful to express the scaling in units of $\mu=\Exp (X)$):


\begin{Def}[Matrix-Pareto distribution]
Let $X\sim \mbox{PH}(\vect{\pi},\mat{T})$ and define
\[  Y = \frac{\beta (e^{X}-1)}{\mu } , \]
where $\beta\in\mathbb{R}$. If we parametrize with
\[  \mu = \Exp (X) = \vect{\pi}(-\mat{T})^{-1}\vect{e} ,\]
 we say that the distribution of $Y$ follows a matrix--Pareto distribution, and write
\[  Y \sim \mbox{M--Pareto}(\beta,\vect{\pi},\mat{T})  .\]
\end{Def}
If $ Y \sim \mbox{M--Pareto}(\beta,\vect{\pi},\mat{T})$, then 
\[  Y \sim \frac{\beta}{\mu} Z \]
where $Z=e^X-1$. Then from Theorem \ref{th:F-and-f-of-log-PH} we get that  
\[ \Prob (Y>y)
= \vect{\pi} \left( 1+ \mu \frac{y}{\beta} \right)^{\mat{T}}\vect{e}  \]
and 
\[  f_Y(y)=\frac{\mu}{\beta}\vect{\pi}(1+\mu \frac{y}{\beta})^{\mat{T}-\mat{I}}\vect{t} .\]
Moreover, by a simple transformation of the formula in Theorem \ref{th:F-and-f-of-log-PH}, we get that 
 $Y$ has Laplace transform 
\[  L_Y(s)= e^{\beta s/\mu}\vect{\pi}\left( \frac{\mu}{s\beta} \right)^{\mat{T}}\Gamma (\mat{T},s\beta/\mu)\vect{t} . \]
Next consider $\Prob (Y>x+y|Y>x)$. From Theorem \ref{th:gen-vershoot} we then get that

\begin{eqnarray*}
\Prob (Y>x+y|Y>x)
&=&\frac{\vect{\pi} \left(1 +\mu \frac{x}{\beta}  \right)^{\mat{T}} }{\vect{\pi} \left(1 +\mu \frac{x}{\beta}  \right)^{\mat{T}} \vect{e}}\cdot \left[ 
\frac{ 1+\mu \frac{x}{\beta}+ \mu \frac{y}{\beta}}{1+\mu \frac{x}{\beta}} \right]^{\mat{T}}\vect{e} \\
&=& \frac{\vect{\pi} \left(1 +\mu \frac{x}{\beta}  \right)^{\mat{T}} }{\vect{\pi} \left(1 +\mu \frac{x}{\beta}  \right)^{\mat{T}} \vect{e}}\cdot \left[ 
1+\frac{\mu }{\beta+\mu x} y \right]^{\mat{T}}\vect{e},
\end{eqnarray*}
from which we see that conditionally on $Y>x$, $Y$ has a matrix--Pareto distribution given by
\[  \mbox{M--Pareto}\left( \beta +\mu x,\frac{\vect{\pi} \left(1 +\mu \frac{x}{\beta}  \right)^{\mat{T}} }{\vect{\pi} \left(1 +\mu \frac{x}{\beta}  \right)^{\mat{T}} \vect{e}},\mat{T}   \right) . \] 
We collect the previous results in the following theorem. 

\begin{Th}\label{th:main-Pareto}
Let $Y\sim \mbox{M--Pareto}(\beta,\vect{\pi},\mat{T})$. Denote by $
\bar{F}(y)=\bar{F}_{\beta,\vect{\pi},\mat{T}}(y)=\Prob (Y>y)$ the survival function of $Y$.
Then 

\medskip

\noindent (a) $\bar{F}_{\beta,\vect{\pi},\mat{T}}(y)= \vect{\pi} \left( 1+ \mu \frac{y}{\beta} \right)^{\mat{T}}\vect{e} $ .

\medskip

\noindent (b) $ f_Y(y)=\frac{\mu}{\beta}\vect{\pi}(1+\mu \frac{y}{\beta})^{\mat{T}-\mat{I}}\vect{t} $. 

\medskip

\noindent (c) $\Prob (Y>x+y|Y>x) = \bar{F}_{\beta +\mu x, \tilde{\pi},\mat{T}}(y)$,
where
\[ \tilde{\vect{\pi}}= \frac{\vect{\pi} \left(1 +\mu \frac{x}{\beta}  \right)^{\mat{T}} }{\vect{\pi} \left(1 +\mu \frac{x}{\beta}  \right)^{\mat{T}} \vect{e}} .\]

\medskip

\noindent (d) $L_Y(s)= e^{\beta s/\mu}\vect{\pi}\left( \frac{\mu}{s\beta} \right)^{\mat{T}}\Gamma (\mat{T},s\beta/\mu)\vect{t} $. 

\medskip 

\noindent (e) If the real part of the eigenvalue for $\mat{T}$ with maximum real part is less than $-\alpha$, then $$\Exp ((\frac{\mu}{\beta}Y+1)^\alpha) =  \vect{\pi}\left( -\alpha\mat{I}-\mat{T} \right)^{-1}\vect{t} <\infty.$$
\end{Th}

\begin{Rem}\rm
Note that these results can indeed be seen as generalizations of results from Theorem 3.4.13 of \cite{ekt}, which stated such properties for the generalized Pareto distribution (which for $\xi>0$ refers to the scalar case $p=1$ in our setting). In particular, property (c) above (cf.\ also \cite[Eqn.(7)]{ahn2012}) is a considerable extension of \cite[Equ.3.53]{ekt}, which can be quite useful when using models based on matrix-Pareto distributions. For instance, one reason for the popularity of Pareto distributions among heavy-tailed distributions in reinsurance modelling are the simple pricing formulas for XL reinsurance contracts based on \cite[Equ.3.53]{ekt}, where $x$ will typically refer to the retention level. The present extension shows that these properties can be carried over to the much more general class of matrix-Pareto distributions.  
\qed
\end{Rem}


%

\begin{Ex}\label{ex:erlang}\rm
Consider the distribution of
\[  Y=e^X-1, \]
where $X$ has an Erlang distribution $\mbox{Er}_n(\lambda)$, i.e. $X\sim X_1+...+X_n$ for i.i.d. $X_i\sim$ Exp$(\lambda)$. Then $Y\sim \mbox{M-Pareto}(\beta,\vect{\pi},\mat{T})$ with $\beta=\mu=\Exp (Y)$ (which is a particular type of a log-Gamma distribution). A phase--type representation of the $\mbox{Er}_n(\lambda)$ distribution is $\vect{\pi}=(1,0,...,0)$ and
\[ \mat{T} = \begin{pmatrix}
-\lambda & \lambda & 0 & ... & 0 & 0 \\
0 & -\lambda & \lambda & ... & 0 & 0 \\
0 & 0 & -\lambda & ... & 0 & 0 \\
\vdots & \vdots &\vdots  & \vdots \vdots \vdots  &\vdots &\vdots  \\
0 & 0 & 0 & ... & -\lambda & \lambda \\
0 & 0 & 0 & ... & 0 & -\lambda
\end{pmatrix} ,\]
where the vector is $n$--dimensional and $\mat{T}$ is an $n\times n$ matrix.
Now 
\[ (s\mat{I}-\mat{T})^{-1} = 
\begin{pmatrix}
\frac{1}{s+\lambda} & \frac{\lambda}{(s+\lambda)^2} &  \frac{\lambda^2}{(s+\lambda)^3} & ... &  \frac{\lambda^{n-1}}{(s+\lambda)^n} \\
0 & \frac{1}{s+\lambda} & \frac{\lambda}{(s+\lambda)^2} &  ... & \frac{\lambda^{n-2}}{(s+\lambda)^{n-1}} \\
0 & 0 &  \frac{1}{s+\lambda} & ... & \frac{\lambda^{n-3}}{(s+\lambda)^{n-2}} \\
\vdots & \vdots & \vdots & \vdots\vdots\vdots & \vdots \\
0 & 0 & 0 & ... & \frac{1}{s+\lambda}
\end{pmatrix},
 \]
 so
\[ (1+x)^{\mat{T}} = \frac{1}{2\pi \ii} \oint_\gamma (1+x)^s (s\mat{I}-\mat{T})^{-1}ds \]
is a matrix with $(i,j)$th element, $j\geq i$, equal to
\[  \frac{1}{2\pi\ii}\oint_\gamma  \frac{\lambda^{j-i}(1+x)^s}{(s+\lambda)^{j-i+1}}ds . \]
Here $\gamma$ is a simple path about the multiple eigenvalue $-\lambda$. For $j<i$ the elements are clearly zero. Defining 
\[  \phi (s) = (s+\lambda)^{j-i+1}\frac{\lambda^{j-i}(1+x)^s}{(s+\lambda)^{j-i+1}} = \lambda^{j-i}(1+x)^{s}  , \]
we get by the residue theorem
\[  \frac{1}{2\pi\ii}\oint_\gamma \frac{\lambda^{j-i}(1+x)^s}{(s+\lambda)^{j-i+1}}ds = \frac{\phi^{(j-i)}(-\lambda)}{(j-i)!} ,  \]
where $\phi^{(k)}(s)$ denotes the $k$th order derivative of $\phi$. But
\[ \frac{d^k}{ds^k}(1+x)^s=(1+x)^s (\log(1+x))^k  \]
so we conclude that
\[   \frac{1}{2\pi\ii}\oint_\gamma  \frac{\lambda^{j-i}(1+x)^s}{(s+\lambda)^{j-i+1}}ds = (1+x)^{-\lambda} \frac{\lambda^{j-i}}{(j-i)!}(\log(1+x))^{j-i}  . \]
Then the tail of $Y$ is seen to be the sum of the terms in the first row of $(1+y)^{\mat{T}}$,
\begin{eqnarray*}
\Prob (Y>y)&=&\vect{\pi}(1+y)^{\mat{T}}\vect{e} \\
&=& (1+y)^{-\lambda}\sum_{j=0}^{n-1}\frac{\lambda^j}{j!}(\log(1+y))^{j} .
\end{eqnarray*}
The density is given by
\begin{eqnarray}
f_Y(y)&=&\vect{\pi}(1+y)^{\mat{T}-\mat{I}}\vect{t} \nonumber \\
&=&(1+y)^{-1}\vect{\pi}(1+y)^{\mat{T}}\vect{t} \nonumber\\
&=& \frac{\lambda^n}{(n-1)!}(1+y)^{-\lambda-1}\left[\log(1+y)\right]^{n-1},  \label{dens-log-erlang}
\end{eqnarray}
since $\vect{t}=(0,0,...,0,\lambda)$, so the density is essentially given by $\lambda$ times the $(1,n)$th element of $(1+x)^{\mat{T}}$. \qed
\end{Ex}
An immediate consequence of Example \ref{ex:erlang} is that if $X$ is a mixture of Erlang distributions $\mbox{Er}_{n_i}(\lambda_i)$, $i=1,...,N$, i.e.
\[  f_X(x)=\sum_{i=1}^N \alpha_i f_{X_i}(x) = 
\sum_{i=1}^N \alpha_i \frac{(x-1)^{n_i-1}}{(n_i-1)!}\lambda_i^{n_i} e^{-\lambda_i x} \]
where $\alpha_1+...+\alpha_N=1$, then $Y$ is a mixture of distributions on the form  \eqref{dens-log-erlang}, i.e.
\[ f_Y(y)=\sum_{i=1}^N \alpha_i\frac{\lambda^{n_i}}{(n_i-1)!}(1+y)^{-\lambda_i-1}\left[\log(1+y)\right]^{n_i-1}, \]
cf. also Theorem \ref{triv}. Mixtures of Erlang distributions constitute the smallest and simplest sub--class of phase--type distributions which is dense in the class of distributions on the positive real line, and has often been employed in applications (see e.g.\ \cite{willmot2007class}). Since $f: x\rightarrow \exp (x)-1$ and $f^{-1}: x\rightarrow \log (x+1) $  are continuous functions, by the continuous mapping theorem it is clear that also the class of mixtures of log--Erlang distributions is dense in the class of distributions on the positive half-line (this denseness was also already noted in \cite[Th.1]{ahn2012} for general log-PH distributions). Due to its potential for modelling purposes, we formulate the above result as a theorem.
\begin{Th}\label{th:denseness-M-Pareto}
The class of matrix--Pareto distributions generated by mixtures of Erlang distributions
\[  f_X(x)= 
\sum_{i=1}^N \alpha_i\frac{(x-1)^{n_i-1}}{(n_i-1)!}\lambda_i^{n_i} e^{-\lambda_i x} \]
contains densities of the form
\[ f_Y(y)=\sum_{i=1}^N \alpha_i\frac{\lambda^{n_i}}{(n_i-1)!}(1+y)^{-\lambda_i-1}\left[\log(1+y)\right]^{n_i-1}. \]
The class is dense (in the sense of weak convergence) in the class of distributions on the positive real line.
\end{Th} 
This means that any distribution may be approximated arbitrarily closely by either a mixture of Erlang distributions or a mixture of matrix--Pareto distributions generated by Erlangs. 
Whether one chooses one over the other class when approximating a distribution in practice should depend on the tail of the distribution to be approximated. If the tail is of Pareto or log--Gamma type, then one will obtain a much better approximation using the matrix--Pareto,  as for the mixtures of Erlangs the number of phases needed would be large and still the resulting approximation will not capture the tail behaviour well. We illustrate this important point in Section \ref{sec:Modelling} by fitting a general matrix--Pareto distribution to a heavy-tailed dataset, and it will turn out that an excellent fit is in fact obtained by a matrix--Pareto which is close to being generated by a mixture based on Erlangs.

\begin{Ex}\rm
Consider the distribution of
\[  Y=e^X-1, \]
where $X$ has a generalized Erlang distribution $\mbox{Er}_n(\lambda_1,...,\lambda_n)$, i.e. $X\sim X_1+...+X_n$ for i.i.d. $X_i\sim \mbox{Exp} (\lambda_i)$, $i=1,...,n$. Assume for simplicity that all $\lambda_i$ are different. Then a similar calculation as above reveals that the density for $Y$ is given by
\[ f_Y(y)=\left(\prod_{j=1}^{n}\lambda_j\right) \sum_{i=1}^n \frac{(1+y)^{-\lambda_i-1}}{\prod_{j\neq i}(\lambda_j-\lambda_i)} .  \] \qed
\end{Ex}
\begin{Ex}\rm
While the derivation of Theorem \ref{th:main-Pareto}(c) basically relies on a probabilistic argument using transition probabilities, most distributional properties
in our construction rely only on the matrix--exponential forms of the phase--type distributions. To emphasize this point, 
we here consider a distribution which can be written in a matrix--exponential form but which is not phase--type. For such distributions the matrix--exponentials are no longer transition matrices.
Distributions whose density can be written in terms of a matrix--exponential, including the phase--type distributions, are known as matrix--exponential distributions (see \cite{bladt2017matrix} for further details) and consist exactly of the distributions with rational Laplace transform. Concretely, 
\[ f(x) = \frac{101}{100} e^{-x}(1-\cos (10x)) \]
is an example of a density that does not belong to a random variable of phase--type but has rational Laplace transform
\[  L_f(s)= \frac{101}{s^3+3s^2+103s+101} \]
and can be written in matrix--exponential form as
\[  f(x)=(101,0,0) \exp \left( 
\begin{pmatrix}
0 & 1 & 0 \\
0 & 0 & 1 \\
-101 & -103 & -3
\end{pmatrix} x
  \right)  
\begin{pmatrix}
0 \\
0\\
1
\end{pmatrix}
 . \]
 Let $\vect{\pi}=(101,0,0)$, $\vect{t}=(0,0,1)^\prime$ and let $\mat{T}$ denote the matrix inside the exponential.
Define the resolvent by 
 \[ \mat{R}(s)=(s\mat{I}-\mat{T})^{-1} 
  \]
and let 
\[ Y=e^{X}-1 . \]
Then the density for $Y$ is given by
\[  f_Y(y)=\vect{\pi}(1+y)^{\mat{T}-\mat{I}}\vect{t} , \]
and since only the first and last elements of $\vect{\pi}$ and $\vect{t}$ are different from zero, respectively, only the $(1,3)-$element in $\mat{R}$ is of importance. Since
\[ \mat{R}(1,3) = \frac{1}{(s+1)(s+1+10\ii)(s+1-10\ii)}, \]
the density for $Y$ is readily calculated (using the residue theorem) to be
\[  f_Y(y)=\frac{101}{100}(1+y)^{-2} -\frac{101}{200}(1+y)^{-2-10\ii}-\frac{101}{200}(1+y)^{-2+10\ii},  \]
which simplifies to
\[  f_Y(y)=\frac{101}{100}(1+y)^{-2}\left(1- \cos \left(10\log(1+y) \right)  \right) . \]
In Figure \ref{ME-ex}, both $f$ and $f_Y$ are plotted. As might be expected, one observes that the shape of the distribution is somewhat preserved while the tail is stretched out.
\begin{figure}
  \centering
  \includegraphics[scale=0.80]{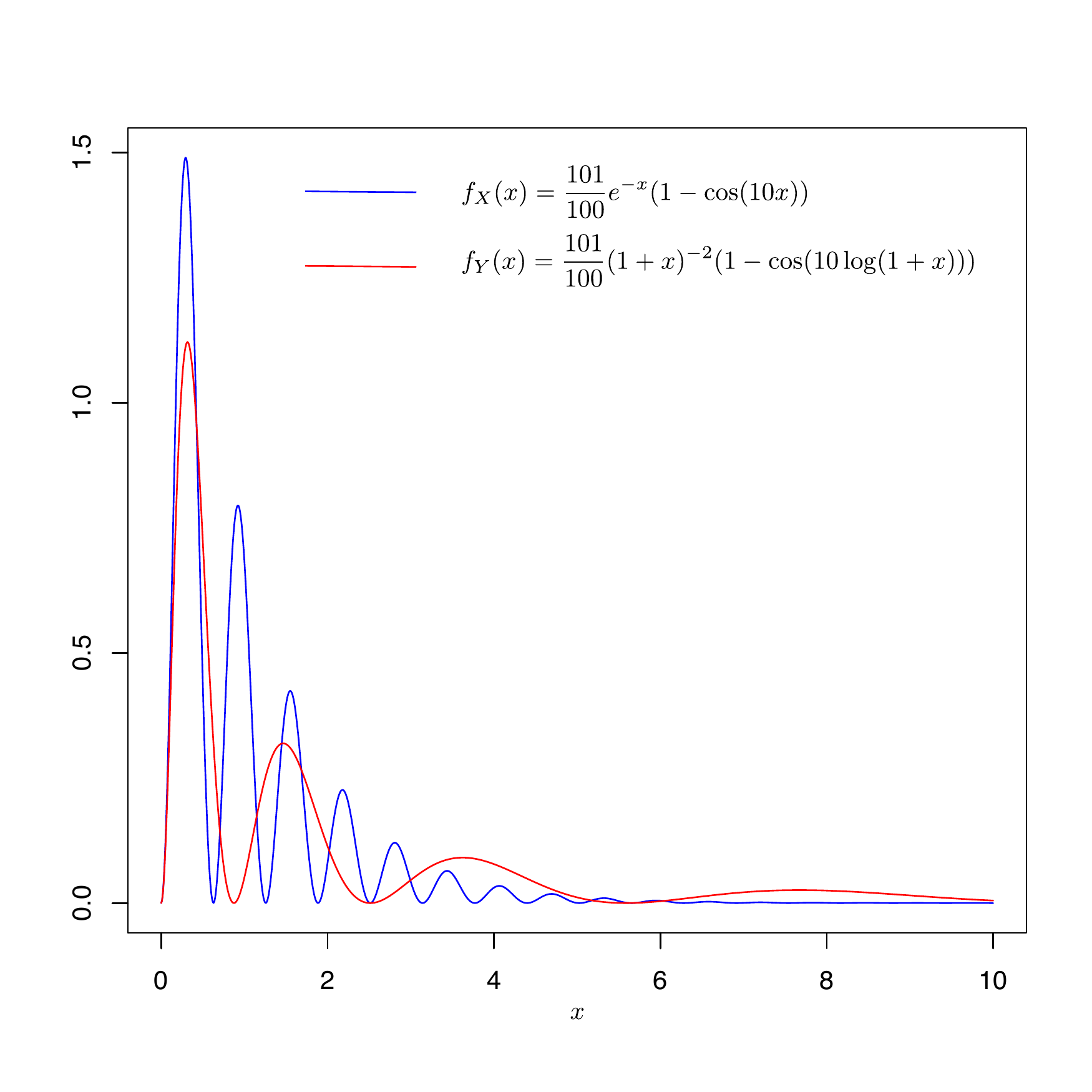}
  \caption{The density for a matrix--exponentially distributed random variable $X$ (blue) and the corresponding density for $Y=\exp(X)-1$ (red).}
  \label{ME-ex}
\end{figure}\qed
\end{Ex}
\section{Analogous generalisations of other classes of distributions}\label{sec:M-Gumbel}
We now list a number of distributions obtained through a transformation of a PH random variable other than the exponential function. All the these transformations, as opposed to the M--Pareto distribution, are parameter--dependent, and as such fitting procedures to data are less straightforward.


\subsection{Matrix-Weibull distributions}\label{eq:power-PH}
Let $X\sim \mbox{PH}_p(\vect{\pi},\mat{T})$. Inspired by the construction of a Weibull random variable as a power of an exponentially distributed random variable, one can define
\[  Y = X^{1/\beta}  \]
for some $\beta >0$. The distribution function of $Y$ is immediately seen to be
\[ F_Y(y) = 1- \vect{\pi}e^{\mat{T}y^\beta}\vect{e} \]
with corresponding density
\[ f_Y(y)=\vect{\pi}e^{\mat{T}y^\beta}\vect{t} \beta y^{\beta-1} . \]
$Y$ can be called a \textit{matrix-Weibull distribution}, since the scale parameter of the usual Weibull distribution with cumulative distribution function $F_{Wei}(y)=1-\exp(-cy^{\beta})$ is now replaced by a matrix. Alternatively, one may refer to $Y$ as a \textit{power-PH distribution}. Its mean is given by
\begin{eqnarray*}
\Exp (Y)&=&\int_0^\infty y\vect{\pi}e^{\mat{T}y^\beta}\vect{t} \beta y^{\beta-1}dy \\
&=& \vect{\pi} \int_0^\infty x^{1/\beta} e^{\mat{T}x}dx\; \vect{t} \\
&=&\Gamma (1+1/\beta)\vect{\pi} (-\mat{T})^{-1/\beta-1} \vect{t},
\end{eqnarray*}
where the last line follows from Theorem 3.4.4 of \cite{bladt2017matrix}. More generally, the  $\theta$th moment ($\theta>0$) can be deduced similarly to be
\[ \Exp (Y^\theta) = \Gamma (1+\theta/\beta)\vect{\pi}(-\mat{T})^{-\theta/\beta-1}\vect{t} .  \]
The moment generating function is found by expansion,
\begin{eqnarray*}
   \Exp (e^{\theta Y}) &=& \int_0^\infty e^{\theta y}\vect{\pi}e^{\mat{T}y^\beta}\vect{t}\beta y^{\beta-1}dy \\
   &=&\vect{\pi}
   \sum_{n=0}^\infty  \frac{\theta^n}{n!}\int_0^\infty x^{n/\beta}e^{\mat{T}x}dx \;\vect{t} \\
   &=&
   \sum_{n=0}^\infty  \frac{\theta^n}{n!}\Gamma (1+n/\beta )\vect{\pi}(-\mat{T})^{-n/\beta-1} \vect{t} ,
\end{eqnarray*} 
which is valid for $\beta>1$. 

If $(\vect{\pi},\mat{T})$ is the representation of an $\mbox{Er}_n(\lambda )$--distribution (see Example \ref{ex:erlang}),  then 
 \[  f_Y(y)= \beta\frac{\lambda^{n-1}}{(n-1)!}y^{n\beta-1} e^{-\lambda y^\beta} . \]
Therefore, as for the matrix--Pareto distribution, a dense class of distributions with a Weibull tail can then be defined through mixtures of the form
\[ f(y)= \beta \sum_{i=1}^N \alpha_i \frac{\lambda_i^{n_i-1}}{(n_i-1)!}y^{n_i\beta-1} e^{-\lambda_i y^\beta} . \]
It is evident that this class will be a very appropriate choice for approximating distributions which have tails (suspected to be) close to a Weibull.

\subsection{Exponential--PH distribution}\label{sec:M-Gumbel1}
\begin{Def}
 Let $X\sim \mbox{PH}_p(\vect{\pi},\mat{T})$. Then define an exponential-PH distribution
 with parameters $\mu$, $\sigma$, $\vect{\pi}$ and $\mat{T}$ as the distribution of the random variable
 \[  Y = \mu - \sigma \log (X) . \]
 \end{Def} 
The distribution function of $Y$ is given by
\begin{eqnarray}
F_Y(y)&=&\Prob \left(\log(X)\geq -\frac{y-\mu}{\sigma}\right) \nonumber\\
&=&\Prob \left(X\geq \exp \left( -\frac{y-\mu}{\sigma} \right)\right)\nonumber \\
&=&\vect{\pi}\exp \left(\mat{T} e^{-(y-\mu)/\sigma} \right)\vect{e} \label{that}
\end{eqnarray}
with support on $(-\infty,\infty)$. The density is obtained through differentiation as
\[  f_Y(y)=\frac{1}{\sigma} e^{-\frac{y-\mu}{\sigma}}\vect{\pi}e^{\mat{T}e^{-\frac{y-\mu}{\sigma}}}\vect{t} =\frac{1}{\sigma} \vect{\pi} \exp \left( -\frac{y-\mu}{\sigma}\mat{I} + \mat{T}e^{-\frac{y-\mu}{\sigma}} \right)\vect{t} . \]
 One can interpret \eqref{that} as a matrix extension of the usual Gumbel distribution with cumulutative distribution function $F_{Gu}(y)=\exp(-\exp(-(y-\mu)/\sigma))$, but since here the matrix $\mat{T}$ replaces the constant $-1$ rather than a parameter, it is less obvious to call it a \textit{matrix-Gumbel distribution}, and one may prefer the name \textit{exponential-PH distribution}. 
Concerning the mean we proceed as follows. Since
\[  \Exp (\log (X))= -\gamma -\vect{\pi}\log (-\mat{T})\vect{e}   \]
(use \eqref{eq:gen-mean-PH} or see e.g. \cite{bladt2017matrix} p. 180) we get that
\[ \Exp (Y) = \mu + \sigma \gamma + \sigma \vect{\pi}\log (-\mat{T})\vect{e} , \]
where $\gamma$ is Euler's constant. In particular, if $X\sim \mbox{Exp} (1)$ then $\mat{T}=-1$, the term $\log (-\mat{T})=0$ and $\Exp (Y)= \mu + \gamma \sigma$. 

Concerning the Laplace transform for $Y$,
\[  L_Y(s)=e^{-s\mu} \Gamma (1+s\sigma) \vect{\pi}(-\mat{T})^{s\sigma}\vect{e},  \]
which follows from
\begin{eqnarray*}
L_Y(s)&=&\Exp \left( e^{-s(\mu-\sigma \log (X))}  \right) \\
&=&e^{-\mu s}\Exp (X^{s\sigma}) \\
&=&e^{-\mu s} \Gamma (1+s\sigma)\vect{\pi}(-\mat{T})^{-s\sigma}\vect{e} ,
\end{eqnarray*}
and where the last step follows from Theorem 3.4.6 of \cite{bladt2017matrix} (p. 175).
For the special (scalar) case of $X\sim \mbox{Exp}(1)$, we recuperate the known formula $e^{-\mu s}\Gamma (1+\sigma s)$ of the Laplace transform of the Gumbel distribution. 

If $(\vect{\pi},\mat{T})$ is the representation of an $\mbox{Er}_n(\lambda )$--distribution (see Example \ref{ex:erlang}),  then  
\[  f_Y(y)=\frac{1}{\sigma} e^{-\frac{y-\mu}{\sigma}}
\frac{\lambda^{(n-1)}}{(n-1)!}e^{-(n-1)\frac{y-\mu}{\sigma}} \exp \left( -\lambda e^{-\frac{y-\mu}{\sigma}}  \right),
  \]
so a dense class of distributions with a Gumbel type of tail can be defined in terms of distributions of the form
 \[  f(y)=\frac{1}{\sigma} e^{-\frac{y-\mu}{\sigma}}\sum_{i=1}^N \alpha_i 
\frac{\lambda_i^{(n_i-1)}}{(n_i-1)!}e^{-(n_i-1)\frac{y-\mu}{\sigma}} \exp \left( -\lambda_i e^{-\frac{y-\mu}{\sigma}}  \right) .
  \]


\subsection{Shifted-power--PH distribution}
More generally, inspired by the Generalized Extreme Value distribution $\mbox{GEV}(\mu,\sigma,\xi)$ with cumulative distribution function 
\begin{equation}\label{thesesc}  F(y) =\left\{  
\begin{array}{ll}
\exp \left(  -\left( 1 + \xi \frac{y-\mu}{\sigma} \right)^{-1/\xi} \right), & \xi \neq 0, \\
\exp \left( -\exp \left( -\frac{y-\mu}{\sigma} \right) \right), & \xi = 0,
\end{array}
\right.
  \end{equation}
we see that for $X\sim \mbox{Exp}(1)$ one has
\begin{eqnarray*}
F(y) &=& \Prob \left( X>\left( 1 + \xi \frac{y-\mu}{\sigma} \right)^{-1/\xi} \right) \\
&=&\Prob \left( \frac{X^{-\xi}-1}{\xi} \leq \frac{y-\mu}{\sigma} \right) \\
&=& \Prob \left(  \mu+  \sigma \frac{X^{-\xi}-1}{\xi} \leq y \right). 
\end{eqnarray*}
That is, for $\xi\neq 0$ one can construct $Y\sim \mbox{GEV}(\mu,\sigma,\xi)$ from an exponentially distributed random variable $X\sim \mbox{Exp} (1)$ by letting
\begin{equation}\label{these}
  Y=\mu+  \sigma \frac{X^{-\xi}-1}{\xi} . \end{equation}
Taking the limit $\xi\rightarrow 0$ then leads back to 
\[  Y=\mu-\sigma \log(X),   \]
which is the Gumbel case treated in Section \ref{sec:M-Gumbel1}, so we focus on $\xi\neq 0$ in the sequel. 

Letting now $X\sim \mbox{PH}(\vect{\pi},\mat{T})$ replace the exponentially distributed random variable above, we apply the transformation \eqref{these} for $\xi\neq 0$. 
The cumulative distribution function for $Y$ is then readily obtained as
\begin{equation}\label{these2}  F_Y(y) = \vect{\pi}\exp\left( \mat{T}\left( 1 + \xi \frac{y-\mu}{\sigma} \right)^{-1/\xi} \right)\vect{e}  , \end{equation}
and the density is obtained through differentiation
\begin{eqnarray*}
f_Y(y) &=&\vect{\pi}\exp\left( \mat{T}\left( 1 + \xi \frac{y-\mu}{\sigma} \right)^{-1/\xi} \right)\mat{T}\vect{e} \left( -\frac{1}{\xi}  \right)\left( 1 + \xi \frac{y-\mu}{\sigma} \right)^{-1/\xi-1}\frac{\xi}{\sigma} \\
&=&\frac{1}{\sigma}\vect{\pi}\exp\left( \mat{T}\left( 1 + \xi \frac{y-\mu}{\sigma} \right)^{-1/\xi} \right)\vect{t}\left( 1 + \xi \frac{y-\mu}{\sigma} \right)^{-(1+\xi)/\xi} \\
&=& \frac{1}{\sigma}z(y)^{\xi+1}\vect{\pi}\exp\left( \mat{T}z(y) \right)\vect{t},
\end{eqnarray*}
where
\[ z(y)=\left( 1 + \xi \frac{y-\mu}{\sigma} \right)^{-1/\xi} . \]
Again, the constant -1 in \eqref{thesesc} is replaced by the matrix $\mat{T}$ in \eqref{these2} now, so that one could view the latter as a certain matrix-GEV distribution, but the name \textit{shifted-power--PH distribution} may be considered more appropriate. 
The mean of $Y$ is 
\begin{eqnarray*} 
\Exp (Y)&=&\mu + \frac{\sigma}{\xi}\left( \Exp (X^{-\xi}-1) \right) \\
&=&\mu +  \frac{\sigma}{\xi} \left( \Gamma (1-\xi) \vect{\pi}(-\mat{T})^\xi \vect{e} -1 \right),
\end{eqnarray*}
which follows from Theorem 3.4.6 of \cite{bladt2017matrix} whenever $\xi<1$. Higher order integer moments can be calculated recursively by expanding the power of $Y$ in terms of $X$ and using Theorem 3.4.6 of \cite{bladt2017matrix} again.

If $(\vect{\pi},\mat{T})$ is the representation of an $\mbox{Er}_n(\lambda )$--distribution, then
\[  f_Y(y) = \frac{1}{\sigma}\frac{\lambda^{(n-1)}}{(n-1)!}t(y)^{\xi+n} \exp \left( -
\lambda t(y) \right)  \]
so a dense class of distributions with GEV--type tails can be defined in terms of mixtures 
\[ f(y)= \frac{1}{\sigma}\sum_{i=1}^N \alpha_i \frac{\lambda_i^{(n_i-1)}}{(n_i-1)!}t(y)^{\xi+n_i} \exp \left( -
\lambda_i t(y) \right) .  \]




\section{Modelling with a matrix--Pareto distribution}\label{sec:Modelling}
As mentioned earlier, the denseness of PH distributions in the class of all distributions on the positive half-line makes them an attractive modelling tool. However, traditionally one disadvantage of this class is that by its exponentially decaying tails it may require many states to be a reasonable fit to distributions with a heavy tail (even on finite time intervals). 
From the denseness of matrix--Pareto distributions in the class of all distributions on the positive half-line (Theorem \ref{th:denseness-M-Pareto} or  \cite[Th.1]{ahn2012}), it becomes clear that if one reckons a heavy-tailed distribution to describe a given set of data-points well, then a fitting procedure with matrix--Pareto distributions may lead to a much more parsimonious (and natural) model (quantitatively expressed by a low number of states and a sparse intensity matrix $\mat{T}$) of the underlying distribution yet keeping the advantageous properties of phase-type distributions. In \cite{ahn2012}, a Danish fire insurance dataset was investigated in this direction. Here, in the context of the present paper, we would like to illustrate the potential of the inhomogeneous PH approach proposed in this paper for a set of 1282 Dutch fire insurance claims from the period 2000-2014 (all in excess of 1 mio.\  EUR), which were already studied in detail and modelled by other means in \cite{abt}. Indeed, after testing various models, in \cite[p.107]{abt} a splicing model with a mixed Erlang component for the bulk, a Pareto distribution for the range to the right of that, and another Pareto distribution for the tail was identified to be an appropriate choice for that dataset.
Apart from seeking an adequate fit by a matrix--Pareto distribution, we shall pay special attention to the aforementioned (possibly mixed) Erlang structure and compare it to the splicing method.
 \\
 
In order to fit the totality of the data by one single distribution, we propose alternatively a matrix--Pareto distribution (log--Phase--type), which also contains the Pareto tail as a special case. Thus we must fit a phase--type distribution to the logarithm of the data. Since the log--data have a support contained in $[8.5,20]$, we decide to subtract 8.5 from all log data which makes it possible to reduce the number of phases significantly. Hence the model we consider is 
\[  y_1,y_2,...,y_N \sim \mbox{PH}(\vect{\pi},\mat{T}), \]
where $y_i=\log(x_i)-8.5\ (i=1,\ldots,N=1282)$ are i.i.d.\ transformed data points $x_1,...,x_N$. The implied model for the original generic data $X$ is then
\[  X= e^{8.5}e^{Y}, \ \ Y\sim \mbox{PH}(\vect{\pi},\mat{T}) , \]
i.e. a (scaled) log--phase--type distribution. 

Utilizing an EM-algorithm (cf.\ \cite{asmner}), we use 20 phases to obtain an adequate fit to the data and the result is presented in Figure \ref{fit}.
\begin{figure}[h]
  \centering
  \includegraphics[scale=0.40]{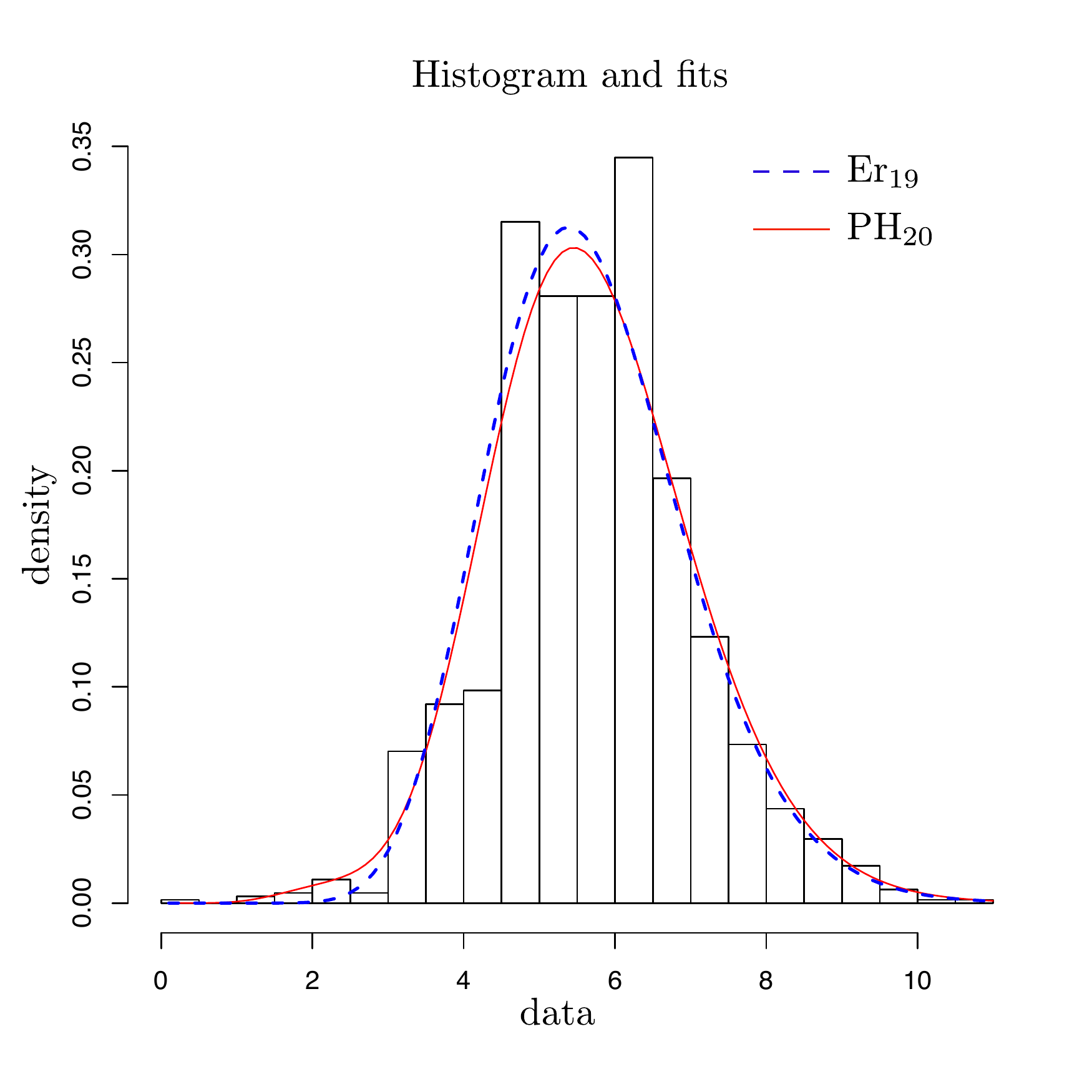}
  \includegraphics[scale=0.40]{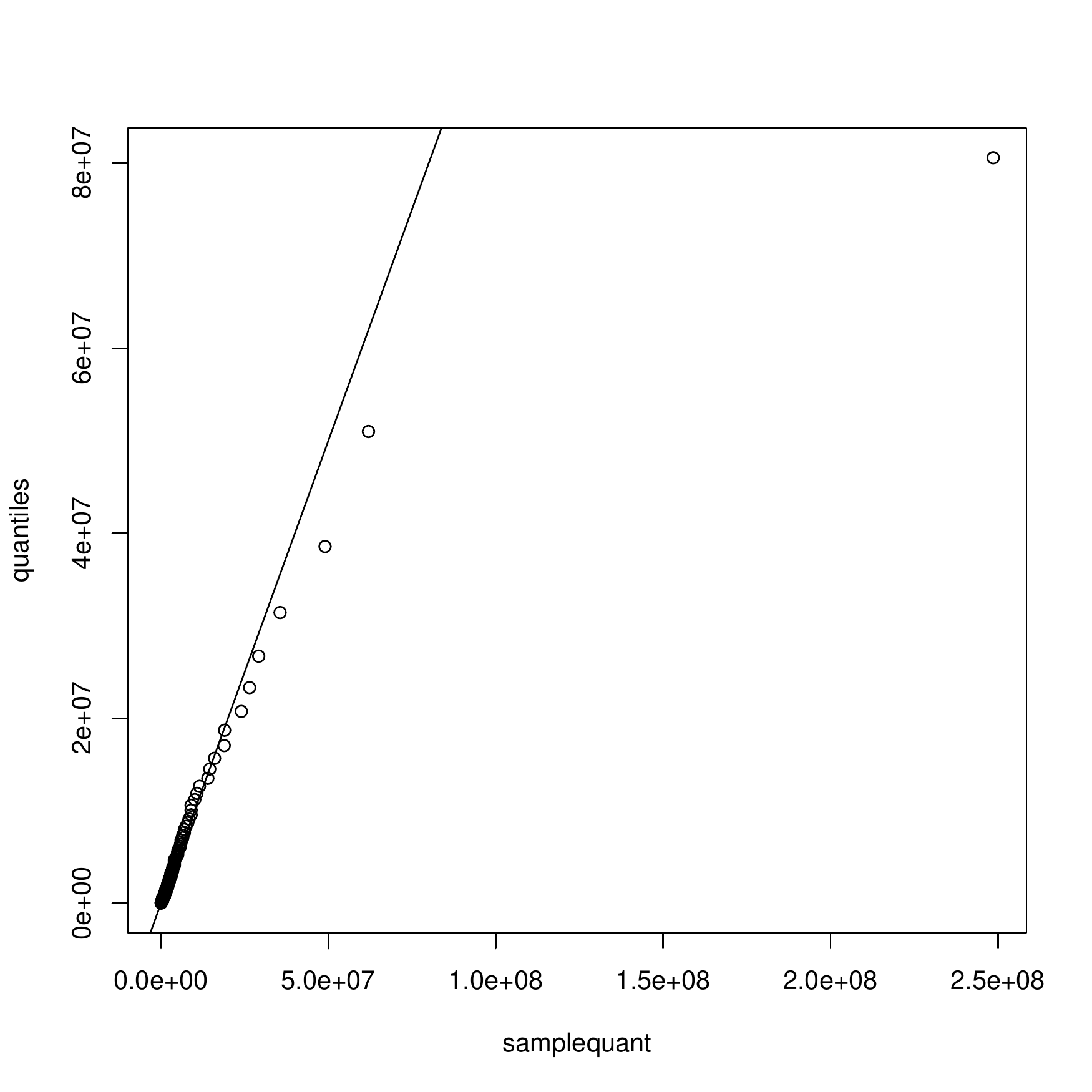}
  \caption{Fitted phase--type distribution to log--data (left) and QQ-plot of original data vs.\ the fitted log--phase--type distribution.}
  \label{fit}
\end{figure}
The choice of 20 phases is taken so as to allow a sufficiently flexible phase--type fit to the data, the concrete magnitude being supported by trial experiments. Though the starting point of the EM algorithm then was a general phase--type distribution, the resulting estimate reduces to a much simpler structure, a phenomenon which is commonly observed in phase--type fitting. Concretely, the structure of the fitted intensity matrix is as follows
\[\tiny
\left(
\begin{array}{rrrrrrrrr|rr|rrrrrrrr|r}
-a & a & 0 & 0 & 0 & 0 & 0 & 0 & 0 & 0 & 0 & 0 & 0 & 0 & 0 & 0 & 0 & 0 & 0 & 0 \\
 0 & -a & a & 0 & 0 & 0 & 0 & 0 & 0 & 0 & 0 & 0 & 0 & 0 & 0 & 0 & 0 & 0 & 0 & 0 \\
0 & 0 & -a & a & 0 & 0 & 0 & 0 & 0 & 0 & 0 & 0 & 0 & 0 & 0 & 0 & 0 & 0 & 0 & 0 \\
 0 & 0 & 0 & -a & a & 0 & 0 & 0 & 0 & 0 & 0 & 0 & 0 & 0 & 0 & 0 & 0 & 0 & 0 & 0  \\
  0  & 0 & 0 & 0 & -a & a & 0 & 0 & 0 & 0 & 0 & 0 & 0 & 0 & 0 & 0 & 0 & 0 & 0 & 0  \\
  0 & 0 & 0 & 0 & 0 & -a & a & 0 & 0 & 0 & 0 & 0 & 0 & 0 & 0 & 0 & 0 & 0 & 0 & 0   \\
  0 & 0 & 0 & 0 & 0 & 0 & -a & a & 0 & 0 & 0 & 0 & 0 & 0 & 0 & 0 & 0 & 0 & 0 & 0   \\
  0 & 0 & 0 & 0 & 0 & 0 & 0 & -a & a & 0 & 0 & 0 & 0 & 0 & 0 & 0 & 0 & 0 & 0 & 0  \\
 0 & 0 & 0 & 0 & 0 & 0 & 0 & 0 & -a & a_1 & 0 & 0 & 0 & 0 & 0 & 0 & 0 & 0 & 0 & 0   \\ \hline
 0 & 0 & 0 & 0 & 0 & 0 & 0 & 0 & 0  & -a_2  & a_3 & 0 & 0 & 0 & 0 & 0 & 0 & 0 & 0 & 0   \\\hline
  0 & 0 & 0 & 0 & 0 & 0 & 0 & 0 & 0 &  0 & -b  & b & 0 & 0 & 0 & 0 & 0 & 0 & 0 &0    \\
  0 & 0  & 0 & 0 & 0 & 0 & 0 & 0 & 0 & 0 &  0 & -b  & b & 0 & 0 & 0 & 0 & 0 & 0 & 0   \\
 0 & 0 & 0  &0 & 0 & 0 & 0 & 0 & 0 & 0 & 0 &  0 & -b  & b & 0 & 0 & 0 & 0 & 0 & 0   \\
  0 & 0 & 0 & 0  & 0 & 0 & 0 & 0 & 0 & 0 & 0 & 0 &  0 & -b  & b & 0 & 0 & 0 & 0 & 0   \\
  0 & 0 & 0 & 0 & 0 & 0 & 0 & 0 & 0 & 0 & 0 & 0 & 0 &  0 & -b  & b & 0 & 0 & 0 & 0   \\
  0 & 0 & 0 & 0 & 0 & 0 & 0 & 0 & 0 & 0 & 0 & 0 & 0 & 0 &  0 & -b  & b & 0 & 0 & 0    \\
  0 & 0 & 0 & 0 & 0 & 0 & 0 & 0 & 0 & 0 & 0 & 0 & 0 & 0 & 0 &  0 & -b  & b & 0 & 0    \\
 0 & 0 & 0 & 0 & 0 & 0 & 0 & 0 & 0 & 0 & 0 & 0 & 0 & 0 & 0 & 0 &  0 & -b  & b & 0    \\
 0 & 0 & 0 & 0 & 0 & 0 & 0 & 0 & 0 & 0 & 0 & 0 & 0 & 0 & 0 & 0 & 0 &  0 & -b  & b_1    \\ \hline
 0 & 0 & 0 & 0 & 0 & 0 & 0 & 0 & 0 & 0 & 0 & 0 & 0 & 0 & 0 & 0 & 0 &  0 & 0  & -c 
\end{array}
\right) ,
\]
where $a=3.322438$, $a_1=3.292191$, $ a_2=3.322884$, $a_3=3.292191$, $b=3.288628$,  $b_1 = 2.063919$  and  $c = 48.21421$ .
The estimator for the initial distribution is 
\[ \vect{\pi} = (0.9992200,0,0,0,0,0,0,0,0,0,0,0,0,0,0,0,0,0,0,0.00078 ) . \]
Hence we may describe this distribution in terms of 8 parameters. The distribution consists of a very fast state (the last one), which can be entered directly with the very small probability 0.00078, representing a possibility to have very small outcomes (one can interpret this to stem from some data points in that region due to the left-truncation of the data at $y_i=0$). Furthermore, there are two blocks of 9--dimensional convolutions of exponential distributions (i.e., Er$_9$ distributions) with intensities $a$ and $b$, respectively. 
Exits to the absorbing state are possible from states 9,10,19 and 20.  

The fit of the phase--type data (cf.\ Figure \ref{fit}) looks adequate in the main body and most of the tail (QQ--plot), which is quite remarkable when compared to the much more complex model suggested in \cite[p.107]{abt}. The seven most extreme tail points start to deviate from the fit, being most pronounced for the last three points. This could be expected since maximum likelihood treats all points of the distribution equally as opposed to extreme value methods focussing on the tail fit. In any case, the identified matrix-Pareto distribution seems to be an excellent and parsimonious fit to the data over quite a large range. 

From the concrete numerical values, one can also see that a simple  matrix-Pareto distribution with underlying Er$_{19}$ distribution (with rate around 3.3, and hence altogether only two parameters(!)) might be a (surprisingly) reasonable description of the data as well. To further pursue this point, we fitted an $\mbox{Er}_{19}(\lambda)$  (i.e.\ a convolution of 19 i.i.d.\ exponential distributions with intensities $\lambda$) to the data resulting in an estimated intensity of $\hat{\lambda}=3.340752$. The Erlang distribution indeed also fits the data well in a first approximation (cf.\ the dashed blue line in Figure \ref{fit}), though it is clear that the full PH is a more suitable model for small values. The log--likelihoods for the full phase--type and Erlang fits are -1250.092 and -1366.735, respectively. 

In conclusion, one may argue that we have obtained a parsimonious, and still somewhat similar model as the one in \cite{abt} but using a quite different rationale, and without the need to decide about the location of splicing points. In that respect, it has the advantage to be rather easily implemented in practice. 

\section{Conclusion}\label{concl}
In this paper we proposed a simple way to carry over the advantages of PH distributions as a modelling tool to distributions with a different tail behaviour. This can be achieved by introducing time-inhomogeneity in the Markov jump process underlying the construction of the PH distribution. In case the respective time scaling is the same for all states of the Markov process, this leads to a number of new simple families of distributions (including matrix-Pareto distributions, (shifted) power-PH distributions, exponential-PH distributions), that all inherit the denseness in the class of distributions on the positive real line, but can lead to much more parsimonious descriptions than a direct PH fit, when the tail of the distribution is in fact a corresponding transform of a PH tail. In view of its relevance in applications, we focused on heavy tails and in particular established various properties of matrix-Pareto distributions. We finally studied the modelling performance of the latter on a real-world dataset from fire insurance, with rather  promising results.

There are several possible directions for future research. One may be to more directly merge light- and heavy-tail fitting by time-scaling only some of the states of the Markov process, and possibly design an EM algorithm that automatically decides on the basis of the dataset how many components of each type are needed for a good description of the data. It should also be feasible to extend the approach proposed in this paper to complement splicing models under censoring, see e.g.\  \cite{reynIme}. Finally, adapting extreme value techniques to the dense distribution classes proposed in this paper may provide interesting alternatives for parsimonious modelling with emphasis on the appropriateness of the fit in the tail, but still good performance for smaller values. \\

\textbf{Acknowledgement.} H.A. acknowledges financial support from the Swiss National Science Foundation Project 200021\_168993.

\bibliographystyle{natbib}
\bibliography{PHBib}

\appendix

\end{document}